\documentclass[10pt]{amsart}						
\usepackage{amsmath}
\usepackage{amsfonts}
\usepackage{amssymb}
	\usepackage[pagebackref,  pdfpagemode=FullScreen,  colorlinks=true]{hyperref}	
	
	\usepackage{verbatim}
	\usepackage{mathrsfs, amssymb}

	\newtheorem{thm}{Theorem}
  	\newtheorem{cor}{Corollary}
  	\newtheorem{lem}{Lemma}

	\theoremstyle{definition}

	\theoremstyle{remark}
  	\newtheorem{rem}{Remark}
	\newtheorem*{ex}{Example}

	\newcommand{\M}{\mathcal{M}}
	\newcommand{\Mbar}{\overline{\mathcal{M}}}
	
	\newcommand{\ZZ}{\mathbb{Z}}

	\makeatletter
	\@namedef{subjclassname@2010}{%
	\textup{2010} Mathematics Subject Classification}
	\makeatother

\begin{document}

\title{Pointed Castelnuovo numbers}

\author{Gavril Farkas}
\address{Humboldt-Universit\"at zu Berlin, Institut f\"ur Mathematik, Unter den Linden 6, 10099 Berlin,\hfill \newline \indent Germany}
\email{farkas@math.hu-berlin.de}

\author{Nicola Tarasca}
\address{University of Utah, Department of Mathematics, 155 S 1400 E, Salt Lake City, UT 84112, USA}
\email{tarasca@math.utah.edu}


\begin{abstract}
The classical Castelnuovo numbers count linear series of minimal degree and fixed dimension on a general curve, in the case when this number is finite. For pencils, that is, linear series of dimension one, the Castelnuovo numbers specialize to the better known Catalan numbers. Using the  Fulton-Pragacz determinantal formula for flag bundles and combinatorial manipulations, we obtain a compact formula for the number of linear series on a general curve having prescribed ramification at an arbitrary point, in the case when the expected number of such linear series is finite. The formula is then used to solve some enumerative problems on moduli spaces of curves.
\end{abstract}

\maketitle

A {\it linear series} of type $\mathfrak{g}^r_d$ on a smooth curve $C$ of genus $g$ is a pair $\ell=(L,V)$ consisting of a line bundle $L$ on $C$ of degree $d$ and a subspace of global sections $V\subset H^0(C, L)$ of projective dimension $r$. The Brill-Noether theorem says that for a general curve $C$, the variety $G^r_d(C)$ of linear series $\mathfrak{g}^r_d$ on $C$ has dimension $\rho(g,r,d) := g-(r+1)(g-d+r)$, and is empty if $\rho(g,r,d)<0$. In particular, when $\rho(g,r,d)=0$ there is a finite number $N_{g,r,d}$ of linear series $\mathfrak{g}^r_d$. This number is equal to
\[
N_{g,r,d} = g!\prod_{i=0}^r \frac{i!}{(g-d+r+i)!}.
\]
Remarkably, Castelnuovo \cite{Castelnuovo} correctly determined $N_{g,r,d}$ in the 1880's using a subtle degeneration argument and Schubert calculus. However, the construction of the moduli space of curves, implicitly assumed in the degeneration, has been achieved only in the 1960's by Deligne and Mumford. A modern rigorous proof of the Brill-Noether theorem appeared in 1980 in the work of Griffiths and Harris \cite{MR563378} and is based on Castelnuovo's original degeneration.

\vskip 4pt

Similarly, one can consider linear series on a general curve having prescribed vanishing at a fixed {\it general} point. For a smooth curve $C$ of genus $g$, let $p\in C$ be a point and $\ell=(L,V)\in G^r_d(C)$. The {\it vanishing sequence} of $\ell$ at $p$
\begin{eqnarray}
 \label{seq}
a^{\ell}(p): 0\leq a_0 < \cdots < a_r \leq d
\end{eqnarray}
is the ordered sequence of distinct vanishing orders of sections in $V$ at the point $p$. Given $r,d$ and a sequence $a: 0\leq a_0 < \cdots < a_r \leq d$ as in (\ref{seq}), the {\it adjusted Brill-Noether number} is defined as
$\rho(g,r,d,a):= \rho(g,r,d)-\sum_{i=0}^r (a_i-i)$.
Eisenbud and Harris (\cite[Proposition 1.2]{MR910206}) proved that a general pointed curve $(C,p)$ of genus $g>0$ admits a linear series $\ell\in G^r_d(C)$ with vanishing sequence $a^{\ell}(p)=a$ if and only if
\begin{eqnarray}
\label{EH}
 \sum_{i=0}^r (a_i-i +g-d+r)_+ \leq g.
\end{eqnarray}
Here $(n)_+:=\max\{n,0\}$ for any integer $n$. Note that this condition is stronger than the condition $\rho(g,r,d,a)\geq 0$.
When (\ref{EH}) is satisfied, the variety of linear series $\ell\in G^r_d(C)$  with vanishing sequence $a$ at the point $p$ is pure of dimension $\rho(g,r,d,a)$.
As in the unpointed case, one can consider the zero-dimensional case.
Let $g,r,d$ be positive integers and $a: 0\leq a_0 < \cdots < a_r \leq d$ as above, such that $\rho(g,r,d, a)=0$.
Then, by (\ref{EH}) the curve $C$ admits a linear series $\mathfrak{g}^r_d$ with vanishing sequence $a$ at the point $p$ if and only if $a_0 +g-d+r\geq 0$. When such linear series exist, their number is counted by the {\it adjusted Castelnuovo number}
\begin{eqnarray}
\label{NN}
N_{g,r,d,a} = g! \frac{\prod_{i<j} (a_j - a_i) }{\prod_{i=0}^r (g-d+r+a_i)!}.
\end{eqnarray}
In order to prove (\ref{NN}), one can specialize the general curve of genus $g$ to a rational curve with $g$ elliptic tails attached to it, specialize the marked point to a point on the rational component, and count via Schubert calculus degenerations of linear series on this singular curve (see the proof of Proposition 1.2 in \cite{MR910206}).

From (\ref{EH}), it follows that if $a$ is the vanishing sequence at a {\it general} point of a linear series $\mathfrak{g}^r_d$ on the general curve, then necessarily $\rho(g,r,d,a)\geq 0$. Moreover, any linear series $\ell\in G^r_d(C)$ on a curve of genus $g=0,1$ satisfies $\rho(g,r,d,a^{\ell}(p))\geq 0$ for {\it any} point $p\in C$.

\vskip 4pt

For $g\geq 2$, pointed curves admitting a linear series with adjusted Brill-Noether number equal to $-1$ at the marked point form a divisor in $\M_{g,1}$, see \cite{MR985853}; when $\rho(g,r,d,a) \leq -2$ this locus has codimension at least $2$ in $\M_{g,1}$. In particular, for a general curve $C$ there exists no linear series $\ell\in G^r_d(C)$ satisfying $a^{\ell}(p)\geq a$ for a point $p\in C$ if $\rho(g,r,d,a) \leq -2$, see \cite{MR3053519}. It follows that for each $\ell\in G^r_d(C)$, the vanishing sequence $a^{\ell}(p)$ at an {\it arbitrary} point $p\in C$  satisfies $\rho(g,r,d,a^{\ell}(p))\geq -1$, and there is at most a finite number of points in $C$ where a linear series $\ell\in G^r_d(C)$ has vanishing sequence $a$ verifying $\rho(g,r,d,a)=-1$. The aim of this note is to determine this number. In the following formula, we let $\delta^i_j$ be the Kronecker delta and set $1/n!=0$, when $n<0$.

\begin{thm}
\label{Thm}
Fix $g\geq 2$ and $a: 0\leq a_0 < \cdots < a_r \leq d$ such that $\rho(g,r,d,{a})= -1$.
For a general curve $C$ of genus $g$, the number of pairs $(p, \ell) \in C\times G^r_d(C)$ such that $a^{\ell}(p)={a}$ is equal to
\begin{align}
\label{nthm}
n_{g,r,d,a}:=
 g! \sum_{0\leq j_1<j_2\leq r } \Bigl( (a_{j_2}-a_{j_1})^2-1  \Bigr) \frac{\prod_{0\leq i<k\leq r}\bigl(a_k-\delta_k^{j_1}-\delta_k^{j_2}-a_i+\delta_i^{j_1}+\delta_i^{j_2}\bigr)}{\prod_{i=0}^r \bigl(g-d+r+a_i-\delta_i^{j_1}-\delta_i^{j_2}\bigr)!}.
\end{align}
\end{thm}

Since $\rho(g,r,d,a)=-1$ and necessarily $\rho(g,r,d)\geq 0$, note that $n_{g,r,d,a}=0$ in the case $a=(0,1,2,\dots,r)$.
The case $r=1$ was previously known. Indeed, up to subtracting a base point, one can suppose that $a_0=0$. Since $\rho(g,1,d,a)=-1$, one has $d\geq \frac{g}{2}+1$ and $a_1=2d-g$.
In Theorem \ref{Thm}, we  recover the following formula from \cite[Theorem B]{MR664324} for the number of pencils vanishing with order $2d-g$ at some unspecified point:
\[
 n_{g,1,d,(0,2d-g)} = (2d-g-1)(2d-g)(2d-g+1)\frac{g!}{d!(g-d)!}.
\]

When $a=(0,1,\dots,r-1,r+1)$ and $\rho(g,r,d)=0$, there is only one non-zero summand in the formula for $n_{g,r,d,a}$. We recover the Pl\"ucker formula for the total number of ramification points on every linear series $\mathfrak{g}^r_d$ on a general curve, see \cite[pg. 345]{MR846932}:
\begin{eqnarray*}
n_{g,r,d,a} =  N_{g,r,d}(r+2)(r+1)r(g-d+r) =  N_{g,r,d}(r+1)\bigl(d+r(g-1)\bigr).
\end{eqnarray*}

\vskip 4pt

Let us consider the next non-trivial example. Suppose $\rho(g,r,d)=n-r-1>0$, and let $s:=g-d+r$. The number of linear series $\ell\in G^r_d(C)$ on a general curve $C$ of genus $g$ satisfying the condition $|\ell(-n\cdot p)|\neq \emptyset$ at a certain unspecified point $p\in C$ is equal to
\begin{eqnarray*}
n_{g,r,d,(0,1,\dots,r-1,n)} = \frac{g!  \cdot n (n^2-1) }{(s-1)! (s+n-1)! (r-1)!} \prod_{i=2}^r \frac{i!\cdot (n-i)}{(s-1+i)!}.
\end{eqnarray*}

\vskip 4pt

Theorem \ref{Thm} is proven in \S \ref{count} using the determinantal formula for flag bundles. The resulting determinant is simplified through a series of combinatorial manipulations.
As an application, we compute classes of closures of pointed Brill-Noether divisors in $\Mbar_{g,1}$ in \S \ref{BNdivclasses}, after a result of Eisenbud and Harris. In \S \ref{codim2} we deduce the non-proportionality of closures of Brill-Noether classes of codimension $2$ in $\Mbar_g$.

\vskip 4pt

We remark that proving Theorem \ref{Thm} via a degeneration argument and Schubert calculus is not feasible. In contrast to the situation from \cite{MR910206} where one computes the numbers $N_{g,r,d,a}$ by specializing to a curve having a rational component and $g$ elliptic tails, here one would have to describe all linear series on elliptic curves having prescribed vanishing at two unspecified points (the exceptional ramification point and the point of attachment to the rest of the curve). However, unlike for $1$-pointed elliptic curves, there is no adequate lower bound for Brill-Noether numbers on arbitrary $2$-pointed elliptic curves. In particular we get a lot more linear series with prescribed ramification than we expect and it is difficult to determine which of these limit linear series are smoothable.


\section{Counting Brill-Noether special points}
\label{count}

Let $C$ be a general curve of genus $g\geq 2$ and fix positive integers  $r$ and $d$, as well as a sequence $$a: 0\leq a_0 < \cdots < a_r \leq d$$  with $\rho(g,r,d,a)=-1$. In this section we count the number $n_{g,r,d,a}$ of pairs $(y,\ell)\in C\times G^r_d(C)$ such that $a^{\ell}(y)=a$. Note that every such linear series is complete.

\vskip 4pt

Let $p$ be a general point of $C$. Choose $m$ such that the line bundle $L\otimes \mathcal{O}_C(mp)$ is non-special for every $L\in {\rm Pic}^d(C)$ (for instance, $m=\max\{2g-2-d+1 ,0\}$). The natural evaluation maps
\[
H^0(L\otimes \mathcal{O}_C(mp)) \rightarrow H^0(L\otimes \mathcal{O}_C(mp)|_{mp+a_r y}) \twoheadrightarrow \cdots \twoheadrightarrow H^0(L\otimes \mathcal{O}_C(mp)|_{mp+a_0 y})
\]
globalize to
\[
\pi^*(\mathcal{E}) \rightarrow \mu_*(\nu^* \mathcal{L}\otimes \mathcal{O}_{D_r})=:\mathcal{M}_r \twoheadrightarrow \cdots \twoheadrightarrow \mu_*(\nu^* \mathcal{L}\otimes \mathcal{O}_{D_0})=:\mathcal{M}_0
\]
as maps of vector bundles over $C\times {\rm Pic}^{d+m}(C)$. Here $\mathcal{L}$ is a Poincar\'e bundle on $C\times {\rm Pic}^{d+m}(C)$, the map $\pi\colon C\times {\rm Pic}^{d+m}(C)\rightarrow {\rm Pic}^{d+m}(C)$ is the second projection, $\mathcal{E}$ is a vector bundle of rank $d+m-g+1$ defined as $\mathcal{E}:=\pi_*(\mathcal{L})$, the maps $\mu\colon C\times C\times {\rm Pic}^{d+m}(C)\rightarrow C\times \mbox{Pic}^{d+m}(C)$ and $\nu\colon C\times C\times {\rm Pic}^{d+m}(C)\rightarrow C\times \mbox{Pic}^{d+m}(C)$  are the projections onto the first and third, and the second and third factors respectively, and finally $\mathcal{O}_{D_i}$ is the structure sheaf of the divisor $D_i$ in $C\times C$ whose restriction to $\{ y\}\times C \cong C$ is $mp+a_i y$.

\vskip 4pt

We are interested in the locus of pairs $(y,L)$ such that $h^0(L\otimes \mathcal{O}_C(-a_i y))\geq r+1-i$, for $i=0,\dots,r$. This is the locus where the morphism of vector bundles
\[
\varphi_i \colon \pi^*(\mathcal{E}) \rightarrow \mathcal{M}_i
\]
has rank at most $d+m+i-g-r$, for $i=0,\dots,r$. The class of this locus can be computed using Fulton-Pragacz determinantal formula for flag bundles
\cite[Theorem 10.1]{MR1154177}.

\vskip 4pt

We shall first compute the Chern polynomial of the bundles $\mathcal{M}_i$. Let $\pi_i:C\times C\times {\rm Pic}^{d+m}(C) \rightarrow C$ for $i=1,2$ and $\pi_3:C\times C\times {\rm Pic}^{d+m}(C) \rightarrow {\rm Pic}^{d+m}(C)$ be the natural projections. Denote by $\theta$ the pull-back to $C\times C\times {\rm Pic}^{d+m}(C)$ of the class $\theta \in H^2({\rm Pic}^{d+m}(C))$ via $\pi_3$, and denote by $\eta_i$ the cohomology class $\pi_i^*([{\rm point}]) \in H^2(C\times C\times {\rm Pic}^{d+m}(C))$, for $i=1,2$. Note that $\eta_i^2=0$. Furthermore, given a symplectic basis $\delta_1,\dots,\delta_{2g}$ for $H^1(C,\ZZ)\cong H^1({\rm Pic}^{d+m}(C),\ZZ)$, we denote by $\delta^i_\alpha$ the pull-back to $C\times C\times {\rm Pic}^{d+m}(C)$ of $\delta_\alpha$ via $\pi_i$, for $i=1,2,3$. Let us define the class
\[
\gamma_{i,j}:= - \sum_{s=1}^{g} \left( \delta^j_s \delta^i_{g+s} -  \delta^j_{g+s} \delta^i_s   \right).
\]
Note that
\[
\begin{array}{cclccccclc}
\gamma_{1,2}^2 &=& -2g \eta_1 \eta_2  & \mbox{ and}  & \eta_i \gamma_{1,2} &=& \gamma_{1,2}^3=0,   &\mbox{ for}  & i=1,2,\\
\gamma_{k,3}^2 &=& -2 \eta_k \theta & \mbox{ and} & \eta_k \gamma_{k,3} &=& \gamma_{k,3}^3=0, & \mbox{ for} & k=1,2,
\end{array}
\]
\[
\gamma_{i,j}\gamma_{j,3}  \quad = \quad  \eta_j \gamma_{i,3}, \quad\quad\mbox{for}\quad \{i,j\}=\{1,2\}.
\]
From \cite[\S VIII.2]{MR770932}, we have
\begin{eqnarray*}
ch(\nu ^* \mathcal{L}) &=& 1+(d+m)\eta_2 + \gamma_{2,3}-\eta_2 \theta,\\
ch(\mathcal{O}_{D_i}) &=& 1-e^{-(a_i \eta_1+a_i \gamma_{1,2}+(a_i+m)\eta_2)},
\end{eqnarray*}
\noindent hence via the Grothendieck-Riemann-Roch formula
\begin{eqnarray*}
ch(\mathcal{M}_i) &=& \mu_* \left(  (1+(1-g)\eta_2) \cdot ch(\nu ^* \mathcal{L} \otimes \mathcal{O}_{D_i}) \right) \\
              &=& a_i+m +\eta_1 (a_i^2(g-1)+a_i(d-g+1))+a_i \gamma_{1,3}-a_i\eta_1\theta.
\end{eqnarray*}
It follows that the Chern polynomial of $\mathcal{M}_i$ is
\[
c_t(\mathcal{M}_i) = 1+\eta_1 (a_i^2(g-1)+a_i(d-g+1))+a_i \gamma_{1,3}+(a_i-a_i^2)\eta_1\theta.
\]
Recall that $c_t(\mathcal{E})=e^{-t\theta}$ (\cite[\S VIII.2]{MR770932}). In the following, we will use the Chern classes $c^{(i)}_t:= c_t (\mathcal{M}_i-\mathcal{E})$, that is,
\[
c^{(i)}_1 = \eta_1 (a_i^2(g-1)+a_i(d-g+1))+a_i \gamma_{1,3} + \theta
\]
and
\[
c^{(i)}_j = \frac{\theta^j}{j!}+\eta_1\theta^{j-1}\left( \frac{a_i^2(g-1)+a_i(d-g+1)}{(j-1)!}+\frac{a_i-a_i^2}{(j-2)!} \right) +\frac{a_i}{(j-1)!}\gamma_{1,3}\theta^{j-1}
\]
for $j\geq 2$.

From the Fulton-Pragacz formula \cite[Theorem 10.1]{MR1154177}, the number of pairs $(y,\ell)$ in $C\times G^r_d(C)$ with $a^\ell(y)=a$ is the degree of the following $(r+1)\times(r+1)$ matrix
\begin{eqnarray}
\label{ndet}
 n_{g,r,d,a}=
\deg \left[
\begin{array}{cccc}
c^{(r)}_{g-d+r+a_r-r} & \cdots & & c^{(r)}_{g-d+r+a_r}\\
c^{(r-1)}_{g-d+r+a_{r-1}-r} & c^{(r-1)}_{g-d+r+a_{r-1}-(r-1)} & \cdots & c^{(r-1)}_{g-d+r+a_{r-1}}\\
\vdots &  & \ddots &\vdots \\
c^{(0)}_{g-d+r+a_0-r} & \cdots & & c^{(0)}_{g-d+r+a_0}
\end{array}
\right].
\end{eqnarray}
Since $\eta_1^2=\eta_1\gamma_{1,3}=\theta^{g+1}=0$, many terms in the expansion of the above determinant are zero. The only terms that survive are the ones obtained by multiplying a summand
\[
 \eta_1\theta^{j-1}\left( \frac{a_i^2(g-1)+a_i(d-g+1)}{(j-1)!}+\frac{a_i-a_i^2}{(j-2)!} \right)
\]
of one of the classes $c^{(i)}_j$ with $r$ summands $\frac{\theta^j}{j!}$ from the other classes $c^{(i)}_j$, or the terms obtained by multiplying two summands
\[
 \frac{a_i}{(j-1)!}\gamma_{1,3}\theta^{j-1}
\]
of two different classes $c^{(i)}_j$ with $r-1$ summands $\frac{\theta^j}{j!}$ from the other classes $c^{(i)}_j$. We use the following variation of the Vandermonde determinant
\begin{eqnarray*}
 \left[
\begin{array}{cccc}
\frac{1}{(b_r-r)!} & \cdots & & \frac{1}{b_r!}\\
\frac{1}{(b_{r-1}-r)!} & \frac{1}{(b_{r-1}-(r-1))!} & \cdots & \frac{1}{b_{r-1}!}\\
\vdots &  & \ddots &\vdots \\
\frac{1}{(b_{0}-r)!} & \cdots & & \frac{1}{b_{0}!}
\end{array}
\right] = \frac{\prod_{l<k} (b_k-b_l)}{\prod_{j=0}^r b_j!}.
\end{eqnarray*}
Hence the quantity (\ref{ndet}) can be written as
\begin{multline}
\label{det}
n_{g,r,d,a}=
\frac{g!}{\prod_{j=0}^r (g-d+r+a_j)!}\\
{}\times \left[\sum_{i=0}^r (a_i^2(g-1)+a_i(d-g+1))(g-d+r+a_i)\prod_{0\leq l< k\leq r}(a_k-\delta^i_k-a_l+\delta^i_l) \right.\\
{}+\sum_{i=0}^r (a_i-a_i^2)(g-d+r+a_i)(g-d+r+a_i-1)\prod_{0\leq l< k\leq r} (a_k-2\delta^i_k-a_l+2\delta^i_l) \\
{} -2\sum_{0\leq i_1<i_2\leq r} a_{i_1} a_{i_2} (g-d+r+a_{i_1})(g-d+r+a_{i_2})\\
\left. \prod_{0\leq l< k\leq r} (a_k-\delta^{i_1}_k-\delta^{i_2}_k-a_l+\delta^{i_1}_l+\delta^{i_2}_l)\right]
\end{multline}
where $\delta^i_j$ is the Kronecker delta.

Remember that $g,r,d,a$ satisfy the condition $\rho(g,r,d,a)=-1$. In the following we use the independent variables $r, a_1, \dots, a_r,$ and $s:=g-d+r$. Note that
\begin{align*}
g &= rs+s-1+\sum_{i=0}^r (a_i-i), &
d &= rs+r-1+\sum_{i=0}^r (a_i-i) .
\end{align*}
Since the right-hand side of (\ref{ndet}) is zero if $a_i=a_j$ for any $i\not=j$, we can write (\ref{det}) as
\begin{eqnarray}
\label{withP}
n_{g,r,d,a}=g!\frac{\prod_{0\leq i<j\leq r} (a_j - a_i)}{\prod_{j=0}^r (g-d+r+a_j)!}\Bigl(P_2(r,a) s^2+ P_3(r,a) s+ P_4(r,a) \Bigr)
\end{eqnarray}
where $P_i(r,a)$ is a polynomial in the variables $r$ and $a_0,\dots,a_r$ which is symmetric in $a_0,\dots,a_r$ for $i=2,3,4$. Note that the expression in the square brackets in (\ref{det}) can be reduced to a linear combination of the following expressions
\begin{eqnarray*}
 \sum_{i=0}^r \, a_i^t \,\prod_{l<k} (a_k -\delta^i_k -a_l +\delta^i_l),\\
 \sum_{i=0}^r \,a_i^t \,\prod_{l<k} (a_k -2\delta^i_k -a_l +2\delta^i_l),\\
 \sum_{i<j} \,(a_i^t a_j^u+a_i^u a_j^t) \,\prod_{l<k} (a_k -\delta^i_k-\delta^j_k -a_l +\delta^i_l+\delta^j_l),
\end{eqnarray*}
for $t,u\geq 0$ such that $t+u\leq 4$. From Lemma \ref{comb1} and Lemma \ref{comb2} (see below), the polynomial $P_i(r,a)$ is symmetric of degree $i$ in $a_0,\dots,a_r$ and has degree at most $i+2$ in $r$, for $i=2,3,4$.

Since the polynomials $P_i(r,a)$ are symmetric in $a_0,\dots,a_r$, they can be expressed in terms of the standard symmetric polynomials in $a_0,\dots, a_r$. That is, we can write $P_i(r,a)$ as a linear combination of the finitely many monomials in
\begin{align*}
\sigma_1 &= \sum_{0\leq i \leq r} a_i, & \sigma_2 &= \sum_{0\leq i < j \leq r} a_i a_j, &
\sigma_3 &= \sum_{0\leq i < j < k \leq r} a_i a_j a_k, & \sigma_4 &= \sum_{0\leq i < j < k <l \leq r} a_i a_j a_k a_l
\end{align*}
of degree at most $i$ in $a_0,\dots,a_r$, with polynomials in $r$ of degree at most $i+2$ as coefficients.
By the bound on the degree in $r$, the polynomial $P_i(r,a)$ is determined by its values at integers $r$ with $1\leq r\leq i+3$.
Hence, the expression in the square brackets in (\ref{det}) is determined by its values at integers $r$ with $1\leq r\leq 7$.


To complete the proof, it remains to verify the equality of the cumbersome expression for $n_{g,r,d,a}$ in (\ref{det}) and the compact expression in (\ref{nthm}).
By pulling out the denominators, the expression in (\ref{nthm}) can be rewritten as follows
\begin{multline}
\label{thm2}
\frac{g!}{\prod_{j=0}^r (g-d+r+a_j)!}\\
 {}\times \left[ \sum_{0\leq j_1<j_2\leq r } \Bigl( (a_{j_2}-a_{j_1})^2-1  \Bigr) (s+a_{j_1})(s+a_{j_2})\prod_{0\leq i<k\leq r}(a_k-\delta_k^{j_1}-\delta_k^{j_2}-a_i+\delta_i^{j_1}+\delta_i^{j_2}) \right].
\end{multline}
Let $f_{s,r,a }$ be the polynomial in the square brackets in (\ref{thm2}), and let $h_{s,r,a}$ be the polynomial in the square brackets in (\ref{det}).
By Lemma \ref{comb2}, formula (\ref{thm2}) can also be written as in (\ref{withP}), with polynomials $P'_i(r,a)$ symmetric of degree $i$ in $a_0,\dots,a_r$ and of degree at most $i+2$ in $r$, for $i=2,3,4$. Hence, to show that (\ref{thm2}) coincides with (\ref{det}), 
it is enough to show that the polynomials $f_{s,r,a}$ and $h_{s,r,a}$ coincide for $1\leq r \leq 7$.
When $r=1$, one has
\begin{eqnarray*}
h_{s,1,a}&=& (a_1-a_0)\Bigl( (\sigma_1^2-4\sigma_2 -1)s^2+ (\sigma_1^3 -4\sigma_1 \sigma_2  - \sigma_1)s +
\sigma_1^2 \sigma_2 -4 \sigma_2^2   -\sigma_2
 \Bigr) =f_{s,1,a}.
\end{eqnarray*}
Thereafter, one verifies the case $r=2$:
\begin{eqnarray*}
h_{s,2,a}&=& \prod_{0\leq i<j\leq 2} (a_j - a_i) \Bigl( (2\sigma_1^2 - 6 \sigma_2 -6)s^2+ (2 \sigma_1^3-7\sigma_1\sigma_2 + 9 \sigma_3
+3\sigma_2 -\sigma_1^2 -4\sigma_1 +3)s \\
&& + \sigma_1^2\sigma_2 -4\sigma_2^2 + 3\sigma_1\sigma_3 -\sigma_1^3 -9\sigma_3
+4\sigma_1\sigma_2 +\sigma_1^2 -5\sigma_2 +\sigma_1 -1  \Bigr)\\
&=& f_{s,2,a},
\end{eqnarray*}
the case $r=3$:
\begin{eqnarray*}
h_{s,3,a}&=& \prod_{0\leq i<j\leq 3} (a_j - a_i) \Bigl(  (3\sigma_1^2  -8\sigma_2 -20)s^2 + (3\sigma_1^3 -10\sigma_1\sigma_2 +12\sigma_3
+8\sigma_2 -3\sigma_1^2 -10\sigma_1 +20)s   \\
&&{}+ \sigma_1^2\sigma_2  -4\sigma_2^2 +3\sigma_1\sigma_3 -3\sigma_1^3 -18\sigma_3 +11\sigma_1\sigma_2
+ 4\sigma_1^2 -14\sigma_2 +5\sigma_1 - 10\Bigr)\\
&=& f_{s,3,a},
\end{eqnarray*}
the case $r=4$:
\begin{eqnarray*}
h_{s,4,a}&=& \prod_{0\leq i<j\leq 4} (a_j - a_i) \Bigl( (4\sigma_1^2 -10\sigma_2 -50)s^2 \\
&&{}+ (4\sigma_1^3 -13\sigma_1\sigma_2
+ 15\sigma_3 +15 \sigma_2 - 6\sigma_1^2 -20\sigma_1 +75 )s \\
&&{} + \sigma_1^2\sigma_2 -4\sigma_2^2 + 3\sigma_1\sigma_3 -6\sigma_1^3 -30\sigma_3 +21\sigma_1\sigma_2
+ 10\sigma_1^2 - 30 \sigma_2 + 15\sigma_1 -50 \Bigr)\\
&=& f_{s,4,a},
\end{eqnarray*}
the case $r=5$:
\begin{eqnarray*}
h_{s,5,a}&=& \prod_{0\leq i<j\leq 5} (a_j - a_i) \Bigl( (5\sigma_1^2 -12\sigma_2 -105)s^2 \\
&&{}+ (5\sigma_1^3 -16\sigma_1\sigma_2
+18\sigma_3 + 24\sigma_2 -10\sigma_1^2 - 35\sigma_1 +210)s\\
&&{}+ \sigma_1^2 \sigma_2 -4\sigma_2^2 +3\sigma_1\sigma_3 -10\sigma_1^3 -45\sigma_3 +34\sigma_1\sigma_2
+20\sigma_1^2 -55\sigma_2 +35\sigma_1 -175 \Bigr)\\
&=& f_{s,5,a},
\end{eqnarray*}
the case $r=6$:
\begin{eqnarray*}
h_{s,6,a}&=& \prod_{0\leq i<j\leq 6} (a_j - a_i) \Bigl( (6\sigma_1^2 -14\sigma_2 -196)s^2 \\
&&{} +(6\sigma_1^3 -19\sigma_1\sigma_2 +21\sigma_3 +35\sigma_2 -15\sigma_1^2 -56\sigma_1 +490)s\\
&&{} +\sigma_1^2\sigma_2 -4\sigma_2^2 + 3\sigma_1\sigma_3 -15\sigma_1^3 -63\sigma_3 +50\sigma_1\sigma_2
+35\sigma_1^2 -91\sigma_2 + 70\sigma_1 -490 \Bigr)\\
&=& f_{s,6,a},
\end{eqnarray*}
and, finally, the case $r=7$:
\begin{eqnarray*}
h_{s,7,a}&=& \prod_{0\leq i<j\leq 7} (a_j - a_i) \Bigl( (7\sigma_1^2 -16\sigma_2 -336)s^2 \\
&&{} +(7\sigma_1^3 -22\sigma_1\sigma_2 +24\sigma_3 +48\sigma_2 -21\sigma_1^2 -84\sigma_1 +1008)s\\
&&{}+\sigma_1^2\sigma_2 -4\sigma_2^2 + 3\sigma_1\sigma_3 - 21\sigma_1^3 -84\sigma_3
+69\sigma_1\sigma_2 +56\sigma_1^2 -140\sigma_2 +126\sigma_1 -1176   \Bigr)\\
&=& f_{s,7,a}.
\end{eqnarray*}
Since $h_{s,r,a}=f_{s,r,a}$ holds for $1\leq r \leq 7$, the formulae (\ref{det}) and (\ref{thm2}) coincide for all $r$. Theorem \ref{Thm} follows.
${}$\hfill$\square$

\vskip1pc

\begin{rem}
We record the values of the polynomials $P_i(r,a)$ appearing in the formula (\ref{withP}):
\begin{eqnarray*}
P_2(r,a) &=& r\sigma_1^2-2(r+1)\sigma_2-\frac{r(r+1)^2(r+2)}{12},\\
P_3(r,a) &=& r\sigma_1^3 - (3r+1)\sigma_1\sigma_2 +3(r+1)\sigma_3 \\
&&{}+ (r^2-1)\sigma_2 -\frac{r(r-1)}{2}\sigma_1^2-\frac{r(r+1)(r+2)}{6}\sigma_1\\
&&{}+\frac{(r-1)r(r+1)^2(r+2)}{24},\\
P_4(r,a) &=& \sigma_1^2 \sigma_2-4\sigma_2^2+3\sigma_1\sigma_3 \\
&&{}-\frac{r(r-1)}{2}\sigma_1^3 -\frac{3r(r+1)}{2}\sigma_3+\frac{(r-1)(3r+2)}{2}\sigma_1 \sigma_2\\
&&{}+\frac{(r-1)r(r+1)}{6}\sigma_1^2 -\frac{r(r+1)(2r+1)}{6}\sigma_2\\
&&{}+\frac{(r-1)r(r+1)(r+2)}{24}\sigma_1 -\frac{(r-1)r^2(r+1)^2(r+2)}{144}.
\end{eqnarray*}
\end{rem}

In the above proof, we have used the following two lemmata.

\begin{lem}
\label{comb1}
We have
\[
 \sum_{i=0}^r \, a_i^t \,\prod_{l<k} (a_k -\delta^i_k -a_l +\delta^i_l) = P(r,a)\prod_{l<k} (a_k  -a_l )
\]
where $P(r,a)$ is a polynomial in $r$ and $a_0,\dots, a_r$, symmetric of degree $t$ in $a_0,\dots, a_r$, and of degree at most $t+1$ in $r$.
\end{lem}

\begin{proof}
It is easy to see that the left-hand side is anti-symmetric in $a_0,\dots, a_r$, hence we can factor by $\prod_{l<k} (a_k  -a_l )$ and obtain a quotient $P(r,a)$  symmetric in $a_0,\dots, a_r$. In particular, any monomial in the variables $a_i$ in the expansion of the left-hand side has degree at least $\frac{r(r+1)}{2}$.

Let us analyze the expansion of the left-hand side. If we first consider only the summands $a_k -a_l$ in each factor of each product, we obtain
\[
 \left(\sum_{i=0}^r  a_i^t \right)\prod_{l<k} (a_k  -a_l ).
\]
This is a homogeneous polynomial in the variables $a_i$ of degree $t+\frac{r(r+1)}{2}$ which contributes the summand $\sum_{i=0}^r  a_i^t$ to $P(r,a)$.

Next, let us consider non-zero summands of type $\delta^i_l -\delta^i_k$ in $j$ factors of each product, and the summands $a_k -a_l$ in the remaining factors of each product, for $1\leq j \leq r$. We obtain
\[
(r+1)\binom{r}{j}
\]
homogeneous polynomials in the variables $a_i$ of degree $t+\frac{r(r+1)}{2}-j$ with coefficients all equal to $1$. The sum of such polynomials, if nonzero, is a homogeneous polynomial in the variables $a_i$ of degree $t+\frac{r(r+1)}{2}-j\geq \frac{r(r+1)}{2}$ with coefficients polynomials in $r$ of degree at most $j+1$. Such polynomial contributes a summand to $P(r,a)$ of degree $t-j$ in the variables $a_i$ and degree at most $j+1$ in $r$ for $j\leq t$, hence the statement.
\end{proof}

\noindent The same result holds for the expressions
\[
 \sum_{i=0}^r \, a_i^t \,\prod_{l<k} (a_k -2\delta^i_k -a_l +2\delta^i_l).
\]
\begin{ex}
 It is easy to verify the following equality
\[
 \sum_{i=0}^r \, a_i \,\prod_{l<k} (a_k -\delta^i_k -a_l +\delta^i_l)= \left(\sum_{i=0}^r a_i - \frac{r(r+1)}{2} \right)\prod_{l<k} (a_k  -a_l ).
\]
\end{ex}
Similarly, we have the following.
\begin{lem}
\label{comb2}
We have
\[
 \sum_{i<j} \,(a_i^t a_j^u+a_i^u a_j^t) \,\prod_{l<k} (a_k -\delta^i_k-\delta^j_k -a_l +\delta^i_l+\delta^j_l) = P(r,a)\prod_{l<k} (a_k  -a_l )
\]
where $P(r,a)$ is a polynomial in $r$ and $a_0,\dots, a_r$, symmetric of degree $t+u$ in $a_0,\dots, a_r$, and of degree at most $t+u+2$ in $r$.
\end{lem}

\section{Classes of pointed Brill-Noether divisors}
\label{BNdivclasses}

As an application of Theorem \ref{Thm}, we compute pointed Brill-Noether divisor classes in $\Mbar_{g,1}$.  We fix a vanishing sequence $a:0\leq a_0<\ldots <a_r\leq d$ such that $\rho(g,r,d,a)=-1$ and let $\M^r_{g,d}\left( a \right)$ be the locus of smooth curves $(C,p)\in \M_{g,1}$ admitting a linear series $\ell \in G^r_d(C)$ having vanishing sequence $a^{\ell}(p)\geq a$.
Eisenbud and Harris proved in \cite[Theorem 4.1]{MR985853} that the class of the closure of a pointed Brill-Noether divisor $\M^r_{g,d}\left(a \right)$ in $\Mbar_{g,1}$ can be expressed as $\mu \mathcal{BN} + \nu \mathcal{W}$, where
\begin{eqnarray}
\label{BN}
\mathcal{BN} := (g+3)\lambda - \frac{g+1}{6}\delta_{\rm irr} - \sum^{g-1}_{i=1} i(g-i) \delta_i
\end{eqnarray}
is the class of the pull-back from $\Mbar_{g}$ of the Brill-Noether divisor,
\[
\mathcal{W} := -\lambda + \binom{g+1}{2} \psi - \sum_{i=1}^{g-1} \binom{g-i+1}{2} \delta_i
\]
is the class of the Weierstrass divisor, and $\mu$ and $\nu$ are some positive rational numbers. We use the method of test curves to find $\mu$ and $\nu$. Let $\delta^i_j$ be the Kronecker delta.

\begin{cor}
\label{divrho-1}
For $g>2$, the class of the divisor $\Mbar^r_{g,d}\left( a \right)$ in $\Mbar_{g,1}$ is equal to
\[
[\Mbar^r_{g,d}\left( a \right)] = \mu \cdot \mathcal{BN} + \nu \cdot \mathcal{W}
\]
where
\[
\mu = - \frac{n_{g,r,d,a} }{2(g^2-1)} + \frac{1 }{4\binom{g-1}{2}}\sum_{i=0}^r n_{g-1,r,d,(a_0+1-\delta^i_0,\dots,a_r+1-\delta^i_r)}
\quad\quad\mbox{and}\quad\quad
\nu = \frac{n_{g,r,d,a} }{g(g^2-1)}.
\]
\end{cor}

\begin{proof}
Let $C$ be a general curve in $\M_{g}$ and consider the curve $\overline{C} =\{[C,y] \}_{y \in C}$ in $\Mbar_{g,1}$ obtained by varying the point $y$ in $C$. The only generator class having non-zero intersection with $\overline{C}$ is $\psi$, and $\overline{C} \cdot \psi = 2g-2$. On the other hand, $\overline{C} \cdot \Mbar^r_{g,d}\left( a \right)$ is equal to the number of pairs $(y,\ell)\in C \times G^r_d(C)$ such that $a^\ell(y)=a$, that is, $n_{g,r,d,a}$.
Hence, we deduce that
\[
\nu = \frac{n_{g,r,d,a} }{(2g-2)\binom{g+1}{2}}.
\]

Furthermore, let $(E,p,q)$ be a two-pointed elliptic curve with $p-q$ not a torsion point in ${\rm Pic}^0(E)$. Consider the curve $\overline{D}$ in $\Mbar_{g,1}$ obtained by identifying the point $q\in E$ with a moving point in a general curve $D$ of genus $g-1$. Then the intersection $\Mbar^r_{g,d}\left(a\right) \cdot \overline{D}$ corresponds to the pairs $(y,\ell)$ where $y$ is a point in $D$ and $\ell=\{\ell_E,\ell_D\}$ is a limit linear series with $a^{\ell_E}(p)=a$. By \cite[Lemma 3.4]{MR985853}, the intersection is everywhere transverse. The only possibility is $\rho(E,p,q)=0$ and $\rho(D,y)=-1$. It follows that $a^{\ell_D}(y)=(a_0+1-\delta^i_0,\dots,a_r+1-\delta^i_r)$, for some $i=0,\dots,r$, and in each case $\ell_E$ is uniquely determined.
Studying the intersection of $\overline{D}$ with the generating classes, we obtain
\[
\sum_{i=0}^r n_{g-1,r,d,(a_0+1-\delta^i_0,\dots,a_r+1-\delta^i_r)}   = \left( \mu (g-1) + \nu \binom{g}{2} \right)(2g-4)
\]
whence we compute $\mu$.
\end{proof}

\begin{ex}
When $r=1$, $d=g-h$, and $a=(0,g-2h)$, we recover the class of the divisor $\Mbar^1_{g,g-h}(a)$ computed by Logan in \cite[Theorem 4.5]{MR1953519}.
\end{ex}

\section{Non-proportionality of Brill-Noether classes of codimension two}
\label{codim2}

In \cite{MR910206} Eisenbud and Harris show that all classes of closures of Brill-Noether divisors in $\Mbar_g$ are proportional. That is, if $\rho(g,r,d)=-1$, then the class of the closure of the locus $\M_{g,d}^r$ of curves with a linear series $\mathfrak{g}^r_d$ is
\[
[\Mbar_{g,d}^r] = c \cdot \mathcal{BN} \in CH^1(\overline{\mathcal{M}}_g),
\]
where the class $\mathcal{BN}$ is in (\ref{BN}), and $c$ is a positive rational number.

\vskip 4pt

If $\rho(g,r,d)=-2$, then the locus $\M_{g,d}^r$ of curves admitting a linear series $\mathfrak{g}^r_d$ is pure of codimension two (\cite{MR985853}). In the case $r=1$, the class of the closure of the Hurwitz-Brill-Noether locus $\M_{2k,k}^1$ has been computed in \cite{MR3109733} using the space of admissible covers.
In this section, we show that classes of Brill-Noether loci of codimension two are generally not proportional in $CH^2(\overline{\mathcal{M}}_g)$.

\vskip 4pt

The first non-trivial case is when $g=10$: in $\M_{10}$ we consider the two Brill-Noether loci $\M_{10,5}^1$ and $\M_{10,8}^2$ of codimension two.
In order to show that the classes of the closures of $\M_{10,5}^1$ and $\M_{10,8}^2$ are not proportional, we show that their restrictions to two test families are not proportional.

For $i=2,3$, let $C_i$ be a general curve of genus $i$, and $C_{g-i}$ a general curve of genus $g-i$. Consider the two-dimensional family $S_i$ of curves obtained by identifying a moving point $x$ in $C_i$ with a moving point $y$ in $C_{g-i}$. The base of this family is $C_i\times C_{g-i}$.

An element $C_i \cup_{x\sim y} C_{g-i}$ of the family $S_i$ is in the closure of $\M_{10,8}^2$ if and only if it admits a limit linear series $\{ \ell_{C_i}, \ell_{C_{g-i}} \}$ of type $\mathfrak{g}^2_8$ such that $\rho(i,2,8,a^{\ell_{C_i}}(x)) =\rho(g-i,2,8,a^{\ell_{C_{g-i}}}(y))=-1$. There are exactly
\begin{eqnarray*}
T_i := \mathop{\sum_{a=(a_0, a_1, a_2)}}_{\rho(i,2,8,a)=-1} n_{i,2,8,a} \cdot n_{g-i,2,8,(d-a_2, d-a_1, d-a_0)}
\end{eqnarray*}
pairs $(x,y)$ in $C_i\times C_{g-i}$ with this property. Moreover, since the family $S_i$ is in the locus of curves of compact type, we known that the intersection is transverse at each point \cite[Lemma 3.4]{MR910206}. Hence, we have
\begin{align*}
S_2 \cdot \left[ \Mbar_{10,8}^2 \right] &= T_2 = 23184, & S_3 \cdot \left[ \Mbar_{10,8}^2 \right] &= T_3 = 48384.
\end{align*}
Similarly, we compute
\begin{align*}
S_2 \cdot \left[ \Mbar_{10,5}^1 \right] &= 2016,   & S_3 \cdot \left[ \Mbar_{10,5}^1 \right] &=  12096.
\end{align*}
Since the restriction of $[\Mbar_{10,8}^2]$ and $[\Mbar_{10,5}^1]$ to the surfaces $S_2$ and $S_3$ are not proportional, we deduce that $[\Mbar_{10,8}^2]$ and $[\Mbar_{10,5}^1]$ are not proportional.

\bibliographystyle{alpha}
\bibliography{Biblio.bib}

\end{document}